\documentclass{amsproc}
\usepackage{amsfonts}

\theoremstyle{plain}

\newtheorem{corollary}{Corollary}

\newtheorem{proposition}{Proposition}
\newtheorem{remark}{Remark}

\newtheorem{theorem}{Theorem}
\theoremstyle{definition}
\newtheorem{example}{Example}

\usepackage[T1]{fontenc}
\DeclareMathOperator\id{id}
\begin{document}
\title[Invariance principle]{ Mean-type mappings and invariance principle}
\author{Janusz Matkowski}
\address{Institute of Mathematics, University of Zielona G\'{o}ra, Szafrana
4a, PL-65-516 Zielona G\'{o}ra, Poland}
\email{j.matkowski@wmie.uz.zgora.pl}
\author{Pawe\l\ Pasteczka}
\address{Institute of Mathematics, Pedagogical University of Krakow, Podchor\k{a}\.{z}ych 2, PL-30-084 Krak\'{o}w, Poland}
\email{pawel.pasteczka@up.krakow.pl}

\begin{abstract}
In the finite dimensional case, mean-type mappings, their invariant \ means,
relations between the uniqueness of invariant means and convergence of
orbits of the mapping, are considered. In particular it is shown, that the
uniqueness of an invariance mean implies the convergence of all orbits. A
strongly irregular mean-type mapping is constructed and its unique invariant
mean is determined. An application in solving a functional equation is
presented.
\end{abstract}

\maketitle

\section{Introduction}

\footnotetext{\textit{2010 Mathematics Subject Classification. }Primary
26E60, 39B22, Secondary 39B12
\par
\textit{Keywords and phrases:} means, mean-type mapping, invariant mean,
functional equation, iteration, orbit
\par
{}
\par
{}}

We deal with mean-type mappings, invariant \ means with respect to the
mean-type mappings, relations between the uniqueness of invariant means and
convergence of the orbits of the mean-type mappings, in general finite
dimensional case.

The main result of section 2, Theorem 1, says that, without any regularity
conditions, the orbits of the mean-type mapping converge if, and only if,
the mean-type map has a unique invariant mean. In particular, the uniqueness
of invariance mean implies the convergence of the orbits, and each
coordinate of the limit mean-type map is just the invariant mean. This
result generalizes the suitable result in \cite{M-P} where two-dimensional
case is considered.

In section 3 we show that the continuity of the mean-type mapping together
with its weak contractivity are sufficient conditions for the uniqueness of
the invariant mean.

In iteration theory of mean-type mappings, the continuity of the invariant
mean was assumed to guarantee its uniqueness (see for example \cite{Borwein,JM1999,JM2009} and \cite[p. 134, Theorem 83]{ToadeCostin}). In
a recent paper \cite{JM2013}, basing on the fact that every mean is
continuous on the main diagonal of its domain, it was shown that this
continuity assumption is redundant.

In section 4, making use of the discontinuous additive functions, we
construct a mean-type mapping, which is discontinuous at every point outside
of the diagonal and, applying Theorem 1, we show that the arithmetic mean is
a unique invariant mean for it.

In section 5, again applying Theorem 1, we find all continuous functions
which are invariant with respect to a given mean-type mapping and we show
that in this case the assumption of continuity is indispensable.

\section{Invariance principle}

Let $I\subset \mathbb{R}$ be an interval and $p\in \mathbb{N},$ be fixed.

A function $M:I^{p}\rightarrow I$ is called a mean in $I$ if it is internal,
that is if%
\begin{equation*}
\min \left( v_{1},...,v_{p}\right) \leq M\left( v_{1},...,v_{p}\right) \leq
\max \left( v_{1},...,v_{p}\right) \text{, \ \ \ \ \ }v_{1},...,v_{p}\in I\,%
\text{, }
\end{equation*}%
or, briefly, if%
\begin{equation*}
\min v\leq M\left( v\right) \leq \max v\text{, \ \ \ \ \ }v\in I^{p}\text{.}%
\,\text{ }
\end{equation*}

In the sequel, to avoid the trivial results, we assume that $p>1$.

A mapping $\mathbf{M}\colon I^{p}\rightarrow I^{p}$ is referred to as 
\emph{mean-type} if there exists some means $M_{i}\colon I^{p}\rightarrow I$%
, $i=1,\dots ,p$, such that $\mathbf{M}=(M_{1},\dots ,M_{p})$.

We say that a function $K\colon I^{p}\rightarrow \mathbb{R}$ is invariant
with respect to $\mathbf{M}$ (briefly $\mathbf{M}$-invariant), if $K\circ 
\mathbf{M}=K$.

Now, following the idea from \cite{P2018}, for a mean-type mapping $\mathbf{M}\colon I^{p}\rightarrow I^{p}$ we define
an orbit $\mathcal{O}_{\mathbf{M}}:I^{p}\rightarrow (I^{p})^{\infty }$ by 
\begin{equation*}
\mathcal{O}_{\mathbf{M}}(v):=\left( v,\mathbf{M}(v),\mathbf{M}^{2}(v),\dots
\right)
\end{equation*}%
where $\mathbf{M}^{n}$ is the $n$-th iterate of $\mathbf{M}$, $n=0,1,...$.
In view of well known isomorphism $(X^{Y})^{Z}\sim X^{Y\times Z}$, the
function $\mathcal{O}_{\mathbf{M}}^{\ast }\colon I^{p}\rightarrow I^{\infty
} $ is given by 
\begin{equation*}
\mathcal{O}_{\mathbf{M}}^{\ast }(v):=\left( v_{1},\dots ,v_{p},[\mathbf{M}%
(v)]_{1},\dots ,[\mathbf{M}(v)]_{p},[\mathbf{M}^{2}(v)]_{1},\dots ,[\mathbf{M%
}^{2}(v)]_{p},\dots \right) ,
\end{equation*}%
where $[\mathbf{M}^{n}(v)]_{i}$ stands for the $i$-th coordinate of the
vector $\mathbf{M}^{n}(v)$, $i\in \{1,\dots ,p\}$.

By $\ell ^{\infty }(I)$ denote the set of all bounded sequences $a=(a_{1},a_{2},%
\dots )$ with values in $I$.

For $p\in \mathbb{N}$ a function $\phi \colon \ell ^{\infty }(I)\rightarrow
I $ is called $\emph{p}$\emph{-limit-like} if for every $a=(a_{1},a_{2},%
\dots )\in \ell ^{\infty }(I)$

(i)$\ \ \ \ \phi (a_{1},a_{2},a_{3},\dots )=\phi
(a_{p+1},a_{p+2},a_{p+3},\dots )$, and

(ii) $\ \liminf_{n\rightarrow \infty }a_{n}\leq \phi (a_{1},a_{2},\dots
)\leq \limsup_{n\rightarrow \infty }a_{n}$.

Note that whenever the sequence $a$ is convergent, then $\phi
(a)=\lim_{n\rightarrow \infty }a_{n}$.

\begin{proposition}
Let $\mathbf{M}\colon I^{p}\rightarrow I^{p}$ be a mean-type mapping, and $%
\phi \colon \ell ^{\infty }(I)\rightarrow I$ be a $p$-limit-like function.
Then the function $\mathcal{M}_{\phi }\colon I^{p}\rightarrow \mathbb{R}$
given by $\mathcal{M}_{\phi }:=\phi \circ \mathcal{O}_{\mathbf{M}}^{\ast }$
is a mean on $I$, which is $\mathbf{M}$-invariant.

Conversely, every $\mathbf{M}$-invariant mean equals $\mathcal{M}_{\phi }$
for some p-limit-like function $\phi $.
\end{proposition}

\begin{proof}
By the definition of mean the sequence $(\max \mathbf{M}^{n}(v))_{n\in 
\mathbb{N}}$ is nondecreasing and 
\begin{equation*}
\limsup \mathcal{O}_{\mathbf{M}}^{\ast }\left( v\right)
=\limsup_{n\rightarrow \infty }\max \mathbf{M}^{n}(v)\leq \max \mathbf{M}%
^{0}(v)=\max (v).
\end{equation*}%
Similarly we obtain $\liminf \mathcal{O}_{\mathbf{M}}^{\ast }(v)\geq \min
(v) $. Now, as $\phi $ is between $\liminf $ and $\limsup $, we obtain that $%
\mathcal{M}_{\phi }$ is a mean. Moreover 
\begin{align*}
\mathcal{M}_{\phi }\circ \mathbf{M}(v)& =\phi \circ \mathcal{O}_{\mathbf{M}%
}^{\ast }(\mathbf{M}(v))=\phi \left( \lbrack \mathbf{M}(v)]_{1},\dots ,[%
\mathbf{M}(v)]_{p},[\mathbf{M}^{2}(v)]_{1},\dots ,[\mathbf{M}%
^{2}(v)]_{p},\dots \right) \\
& =\phi \left( v_{1},\dots ,v_{p},[\mathbf{M}(v)]_{1},\dots ,[\mathbf{M}%
(v)]_{p},[\mathbf{M}^{2}(v)]_{1},\dots ,[\mathbf{M}^{2}(v)]_{p},\dots \right)
\\
& =\phi \circ \mathcal{O}_{\mathbf{M}}^{\ast }(v)=\mathcal{M}_{\phi }(v)
\end{align*}%
which concludes the proof.

To prove the converse, for an arbitrary $\mathbf{M}$-invariant mean $K$, we
define function $\phi $ on the every orbit $\mathcal{O}_{\mathbf{M}}^{\ast
}(v)$ by 
\begin{equation}
\phi (\mathcal{O}_{\mathbf{M}}^{\ast }(v)):=K(v)\qquad v\in I^{p},  \label{1}
\end{equation}

For every $v\in I^{p}$ we have $\phi (\mathcal{O}_{\mathbf{M}}^{\ast
}(v))=K(v)=K\circ \mathbf{M}(v)=\phi (\mathcal{O}_{\mathbf{M}}^{\ast }(%
\mathbf{M}(v)))$ what implies that $\phi $ satisfies (i) on the image $%
\mathcal{O}_{\mathbf{M}}^{\ast }(I^{p})$.

To preserve the $p$-limit-like properly we need to extend this definition to%
\begin{equation}
\phi (\mathcal{O}_{\mathbf{M}}^{\ast }(v)):=K(v)\text{, \ \ \ \ \ }v\in I^{p}
\label{2}
\end{equation}%
\begin{equation}
\phi (I^{cp}\times \mathcal{O}_{\mathbf{M}}^{\ast }(v)):=K(v)\text{, \ \ \ \ 
}v\in I^{p},\,\,c\in \{1,2,\dots \}.  \label{3}
\end{equation}

We underline that if for some $w\in I^{p}$ and $c_{0}> 0$ we get $%
\mathcal{O}_{\mathbf{M}}^{\ast }(w)\in I^{c_{0}p}\times \mathcal{O}_{\mathbf{%
M}}^{\ast }(v)$ then by the definition $\mathbf{M}^{c_{0}}(w)=v$ and,
consequently, $K(v)=K(w)$. Therefore definitions \eqref{1} and (3) are coherent.
Moreover, it is easy to check that the set 
\begin{equation*}
\Gamma :=\bigcup_{v\in I^{p}} \Big(\big\{\mathcal{O}%
_{\mathbf{M}}^{\ast }(v)\big\} \cup \bigcup_{c=1}^{\infty }I^{cp}\times \mathcal{O}%
_{\mathbf{M}}^{\ast }(v) \Big)\subset \ell^\infty(I)
\end{equation*}%
is closed under shifting by $p$ elements (both left and right). Thus so is
the set $\Gamma ^{\prime}:=\ell^\infty (I)\setminus \Gamma $. Furthermore
properties (i) and (ii) are valid on $\phi |_{\Gamma }$.

As the value of $\phi $ on $\Gamma ^{\prime }$ does not affect the value of $%
\mathcal{M}_{\phi }$, we can define it in any way, just to keep validity of
(i) and (ii) e.g. $\phi (a):=\liminf a$ for $a\in \Gamma ^{\prime }$.
\end{proof}

\begin{corollary}
Let $\mathbf{M}\colon I^{p}\rightarrow I^{p}$ be a mean-type mapping. $%
\mathcal{L}:=\mathcal{M}_{\liminf }$ and $\mathcal{U}:=\mathcal{M}_{\limsup }
$ are the smallest and the biggest $\mathbf{M}$-invariant means,
respectively.
\end{corollary}

\begin{remark}
Let us underline that analogous corollary can be also established if $\mathbb{M}$ is a selfmapping of compactly supported Borel measures (see \cite{DP2020} for details).
\end{remark}
 
\begin{theorem}
(\textbf{Invariance Principle}) Let $\mathbf{M}\colon I^{p}\rightarrow I^{p}$
be a mean-type mapping and $K:I^{p}\rightarrow I$ be an arbitrary mean. $K$
is a unique $\mathbf{M}$-invariant mean if and only if the sequence of
iterates $\left( \mathbf{M}^{n}\right) _{n\in \mathbb{N}}$ of the mean-type
mapping $\mathbf{M}$ converges to $\mathbf{K}:=\left( K,\dots ,K\right) $
pointwise on $I^{p}$.
\end{theorem}

\begin{proof}
We have the following equivalent conditions: 
\begin{align*}
& \qquad K\text{ is a unique }\mathbf{M}\text{-invariant mean} \\
\iff & \qquad K=\mathcal{L}=\mathcal{U} \\
\iff & \qquad K(v)=\liminf \mathcal{O}_{\mathbf{M}}^{\ast }(v)=\limsup 
\mathcal{O}_{\mathbf{M}}^{\ast }(v)\text{ for all }v\in I^{p} \\
\iff & \qquad \mathcal{O}_{\mathbf{M}}^{\ast }(v)\text{ is convergent to }%
K(v)\text{ for all }v\in I^{p} \\
\iff & \qquad \mathcal{O}_{\mathbf{M}}(v)\text{ is convergent to }\mathbf{K}%
(v)\text{ for all }v\in I^{p} \\
\iff & \qquad \mathcal{O}_{\mathbf{M}}\text{ is convergent to }\mathbf{K}%
\text{ pointwise on }I^{p} \\
\iff & \qquad \mathbf{M}^{n}\text{ is convergent to }\mathbf{K}\text{
pointwise on }I^{p},
\end{align*}%
thus the proof is complete.
\end{proof}

\section{Weakly contractive mean-type mappings}

We say that a mean-type mapping $\mathbf{M}\colon I^{p}\rightarrow I^{p}$ is 
\emph{weakly contractive} if for every nonconstant vector $v\in I^{p}$ there
is a positive integer $n_0\left( v\right) $ such that 
\begin{equation*}
\max (\mathbf{M}^{n}(v))-\min (\mathbf{M}^{n}(v))<\max (v)-\min (v) \qquad \text{ for all }n \ge n_0(v).
\end{equation*}

Let us emphasize that it is sufficient to verify if the inequality above is valid for $n=n_0(v)$. 
Moreover in a special case $p=2$ it was proved \cite{M-P} that $%
\mathbf{M}$ is weakly contractive if and only if $\mathbf{M}^{2}$ is
contractive. This is not the case here. 

Even for $p=3$ we can construct
weakly contractive mean-type mapping on $I^3$ such that the function $I^p \ni v \mapsto n_0(v)$ is unbounded. 

\begin{example}
Take a continuous and weakly-contractive mean-type mapping $\mathbf{M}_0\colon I^{3}\rightarrow I^{3}$ such that
$$\mathbf{M}_0(v_{1},v_{2},v_{3})=
\begin{cases}
\left( v_{1},v_{2},\tfrac{v_{1}+v_{3}}{2}\right) &\text{if }\left\vert
v_{3}-v_{1}\right\vert = (v_{2}-v_{1})^{2};\\
\left( v_{1},v_{1},v_1\right) &\text{if }2\left\vert
v_{3}-v_{1}\right\vert = (v_{2}-v_{1})^{2},
\end{cases}
$$
and the set $\Lambda:=\{(v_1,v_2,v_3) \in I^3 \colon \left\vert
v_3-v_1\right\vert \geq (v_2-v_1)^{2} \}$. Define a mapping $\mathbf{M}\colon I^{3}\rightarrow I^{3}$ by%
\begin{equation*}
\mathbf{M}(v_{1},v_{2},v_{3}):=\left\{ 
\begin{array}{cc}
\left( v_{1},v_{2},\tfrac{v_{1}+v_{3}}{2}\right)  & \text{if }(v_1,v_2,v_3)\in \Lambda; \\ 
\mathbf{M}_0\left( v_{1},v_{2},v_{3}\right)  & \text{otherwise.}%
\end{array}%
\right. 
\end{equation*}%
Obviosly $\mathbf{M}$ is continuous. Moreover for every $x,y,z \in I$ there exists $n_1(x,y,z)$ such that $\mathbf{M}^{n_1(x,y,z)} \in \Lambda$.
Thus $\mathbf{M}^{n_1(x,y,z)+k}(x,y,z)=\mathbf{M}_0^k(\mathbf{M}^{n_1(x,y,z)})$. Consequently as $\mathbf{M}_0$ is weakly contractive, so is $\mathbf{M}$.

Now take $x \in I$ and $i \in \mathbb{N}$ such that $x+2^{-i}\in I$. 
By simple induction we can describe the \mbox{$\mathbf{M}$-orbit} of the vector $w:=(x,x+2^{-i},x+2^{-i}) \in I^3$. Namely
\begin{equation*}
 \mathbf{M}^{n}(w)= \mathbf{M}^{n}(x,x+2^{-i},x+2^{-i})=
 \begin{cases}
(x,x+2^{-i},x+2^{-i-n})&\qquad \text{ for }n\le i+1 \\
(x,x,x) &\qquad \text{ for }n> i+1.
\end{cases}
\end{equation*}
This equality proves that for all $n < i$ we have 
$$\max (\mathbf{M}^{n}(w))-\min (\mathbf{M}^{n}(w))=\max (w)-\min (w)$$
Thus $n_0(w)\ge i$. As $i$ can be take arbitrary large (obviously $w$ depends on $i$) we have that $n_0$ cannot be bounded.
\end{example}

\begin{theorem}
If $\mathbf{M}:I^p\to I^p$ is a continuous, weakly contractive mean-type
mapping then there exists a unique $\mathbf{M}$-invariant mean $K\colon
I^p\rightarrow I$. Moreover, the sequence of iterates $\left(\mathbf{M}%
^n\right) _{n\in \mathbb{N}}$ converges (pointwise on $I^{p})$ to $\mathbf{K}%
:=\left( K,\dots,K\right) .$
\end{theorem}

\begin{proof}
Assume that $I$ is closed.

First observe that in view of Theorem~1 the moreover part is equivalent to
our assertion. Second, it is sufficient to prove that $\mathcal{L}=\mathcal{U%
}$. Assume to the contrary that $\mathcal{L}(v_{0})\neq \mathcal{U}(v_{0})$
for some $v_{0}\in I^{p}$.

Define a spread $\delta :=\mathcal{U}(v_{0})-\mathcal{L}(v_{0})>0$ and sets 
\begin{align*}
X_{0}& :=\{v\in I^{p}\colon \max (v)-\min (v)\geq \delta \}, \\
X_{k}& :=\{v\in I^{p}\colon \max (v)-\min (v)\in \lbrack \delta ,\delta +%
\tfrac{1}{k}]\}\qquad \text{ for }k\in \mathbb{N}_{+}, \\
X_{\omega }& :=\bigcap_{k=0}^{\infty }X_{k}=\{v\in I^{p}\colon \max (v)-\min
(v)=\delta \}.
\end{align*}%
Then for every $k\in \mathbb{N}$ there exists $n_{k}$ such that $\mathbf{M}%
^{n}(v_{0})\in X_{k}$ for all $n>n_{k}$.

As $X_0$ is compact, there exists a subsequence $(m_k)$ such that $\mathbf{M}%
^{m_k}(v_0) \in X_k$ and the sequence $(\mathbf{M}^{m_k}(v_0))$ is
convergent to some element $w_0 \in I^p$. By the definition, as the
difference $\max \mathbf{M}^{m_k}(v_0)-\min \mathbf{M}^{m_k}(v_0)$ is
nonincreasing, we have $w_0 \in X_\omega$.

As $\mathbf{M}$ is weakly contractive, there exists $s_0 \in \mathbb{N}$
such that

\begin{equation*}
\max \mathbf{M}^{s_{0}}(w_{0})-\min \mathbf{M}^{s_{0}}(w_{0})<\max
(w_{0})-\min (w_{0})=\delta .
\end{equation*}

But $\mathbf{M}$ is continuous, thus there exists an open neighbourhood $W
\ni w_0$ such that

\begin{equation*}
\max \mathbf{M}^{s_{0}}(w)-\min \mathbf{M}^{s_{0}}(w)<\delta \qquad \text{
for all }w\in W.
\end{equation*}

But, by the definition, there exists $k_0 \in \mathbb{N}$ such that $\mathbf{%
M}^{m_{k_0}}(v_0) \in W$. Then

\begin{equation}
\max \mathbf{M}^{s_{0}+m_{k_{0}}}(v_{0})-\min \mathbf{M}%
^{s_{0}+m_{k_{0}}}(v_{0})<\delta .  \label{4}
\end{equation}

On the other hand, as both $\mathcal{L}$ and $\mathcal{U}$ are $\mathbf{M}$%
-invariant means, we have 
\begin{align*}
\min \mathbf{M}^{s_{0}+m_{k_{0}}}(v_{0})& \leq \mathcal{L}\circ \mathbf{M}%
^{s_{0}+m_{k_{0}}}(v_{0})=\mathcal{L}(v_{0}), \\
\max \mathbf{M}^{s_{0}+m_{k_{0}}}(v_{0})& \geq \mathcal{U}\circ \mathbf{M}%
^{s_{0}+m_{k_{0}}}(v_{0})=\mathcal{U}(v_{0})
\end{align*}

which implies 
\begin{equation*}
\max \mathbf{M}^{s_{0}+m_{k_{0}}}(v_{0})-\min \mathbf{M}%
^{s_{0}+m_{k_{0}}}(v_{0})\geq \mathcal{U}(v_{0})-\mathcal{L}(v_{0})=\delta ,
\end{equation*}%
contradicting \eqref{4}.
\end{proof}

\section{Examples of highly discontinuous means}

\begin{example}
For a discontinuous additive function $\alpha :\mathbb{R\rightarrow R}$
define a function $\lambda _{\alpha }:\mathbb{R\rightarrow R}$ by 
\begin{equation*}
\lambda _{\alpha }\left( u\right) :=\frac{3\left\vert \alpha \left( u\right)
\right\vert +3}{4\left\vert \alpha \left( u\right) \right\vert +12}\text{, \
\ \ }u\in \mathbb{R}\text{.}
\end{equation*}%
Since 
\begin{equation}
\frac{1}{4}\leq \lambda _{\alpha }\left( u\right) <\frac{3}{4}\text{, \ \ \
\ \ }u\in \mathbb{R}\text{,}  \label{5}
\end{equation}%
the functions $M,N:\mathbb{R}^{2}\rightarrow \mathbb{R}$ defined by 
\begin{equation*}
M\left( u,v\right) =\lambda _{\alpha }\left( u\right) u+\left( 1-\lambda
_{\alpha }\left( u\right) \right) v\text{, \ \ \ \ \ }M\left( u,v\right)
=\left( 1-\lambda _{\alpha }\left( u\right) \right) u+\lambda _{\alpha
}\left( u\right) v\text{.}
\end{equation*}%
are means in $\mathbb{R}$. Both $M$ and $N$, being the means, are continuous
at every point of the diagonal $\Delta :=\left\{ \left( x,x\right) :x\in 
\mathbb{R}\right\} $ (see \cite{JM2013}) but, as the graph of $\alpha $ is
dense in $\mathbb{R}^{2}$ (see, for instance \cite{Kuczma}), these functions
are strongly irregular in $\mathbb{R}^{2}\backslash \Delta $ (in particular
they are discontinuous at every point outside of $\Delta ).$

Note that 
\begin{equation*}
A\circ \left( M,N\right) =A,
\end{equation*}%
i.e., the arithmetic mean $A:\mathbb{R}^{2}\rightarrow \mathbb{R\,}$, $%
A\left( u,v\right) =\frac{u+v}{2}$, is invariant with respect to the
mean-type mapping $\left( M,N\right) :\mathbb{R}^{2}\rightarrow \mathbb{R}%
^{2}$.

To show that $A$ is a unique $\left( M,N\right) $-invariant mean, first
observe that, \ for all $u,v\in \mathbb{R}$, 
\begin{equation*}
\left\vert M\left( u,v\right) -N\left( u,v\right) \right\vert =\left\vert
2\lambda _{\alpha }\left( u\right) -1\right\vert \left\vert u-v\right\vert 
\text{,}
\end{equation*}%
whence, in view of \eqref{5},%
\begin{equation*}
\left\vert M\left( u,v\right) -N\left( u,v\right) \right\vert \leq \frac{1}{2%
}\left\vert u-v\right\vert \text{, \ \ \ \ \ }u,v\in \mathbb{R}\text{.}
\end{equation*}%
Putting $\left( M_{n},N_{n}\right) :=\left( M,N\right) ^{n}$, $n\in \mathbb{N%
}_{0}$, we hence get 
\begin{equation*}
\left\vert M_{n+1}\left( u,v\right) -N_{n+1}\left( u,v\right) \right\vert
\leq \frac{1}{2}\left\vert M_{n}\left( u,v\right) -N_{n}\left( u,v\right)
\right\vert \text{, \ \ \ \ \ }u,v\in \mathbb{R}\text{, \ }n\in \mathbb{N}%
_{0}\text{,}
\end{equation*}%
whence, by induction,\ 
\begin{equation*}
\left\vert M_{n}\left( u,v\right) -N_{n}\left( u,v\right) \right\vert \leq 
\frac{1}{2^{n}}\left\vert u-v\right\vert \text{, \ \ \ \ \ }u,v\in \mathbb{R}%
\text{, \ }n\in \mathbb{N}\text{.}
\end{equation*}%
This proves that, for every point $\left( u,v\right) \in \mathbb{R}^{2}$,
the orbit 
$$\mathcal{O}_{\mathbf{M}}(\left( u,v\right) )=\left( \left(
M_{n}\left( u,v\right) ,N_{n}\left( u,v\right) \right) :n\in \mathbb{N}%
_{0}\right)$$
approaches the diagonal as $n\rightarrow \infty $. To show the
convergence of the orbit, take arbitrary $\left( u,v\right) $ $\in \mathbb{R}%
^{2}\backslash \Delta $ and put $c=A\left( u,v\right) =$ $\frac{u+v}{2},$ so
we have $v=2c-u$. The invariance of $A$ with respect to $\left( M,N\right) $
implies that\ $N\left( u,v\right) =2c-M\left( u,v\right) $, and, by
induction, $N_{n}\left( u,v\right) =2c-M_{n}\left( u,v\right) $ for all $%
n\in \mathbb{N}_{0}$, that is, every point $\left( M_{n}\left( u,v\right)
,N_{n}\left( u,v\right) \right) $ lays on the straight-line crossing
perpendicularly the diagonal $\Delta $ at the point $\left( c,c\right) $. It
follows that 
\begin{equation*}
\lim_{n\rightarrow \infty }\left( M_{n}\left( u,v\right) ,N_{n}\left(
u,v\right) \right) =\left( c.c\right) .
\end{equation*}
Applying Theorem 1 we conclude that $A$ is a unique $\left( M,N\right) $%
-invariant mean.
\end{example}

The result considered in this example can be easily extended to the following

\begin{proposition}
Let $\alpha :\mathbb{R\rightarrow R}$ be a discontinuous additive function, $%
\kappa \in \left( 0,1\right) $\thinspace and $b,c,d$ be real numbers such
that%
\begin{equation*}
c>1\text{, \ }0<b\leq d\text{, \ \ \ }\frac{2}{1+\kappa }\leq c\leq \frac{2}{%
1-\kappa }\text{, \ \ \ }\frac{2b}{1+\kappa }\leq d\leq \frac{2b}{1-\kappa }%
\text{, }
\end{equation*}%
and let $\lambda _{\alpha }:\mathbb{R\rightarrow R}$ be defined by 
\begin{equation*}
\lambda _{\alpha }\left( u\right) :=\frac{\left\vert \alpha \left( u\right)
\right\vert +b}{c\left\vert \alpha \left( u\right) \right\vert +d}\text{, \
\ \ }u\in \mathbb{R}\text{.}
\end{equation*}%
Then

\ (i) \ the functions $M,N:\mathbb{R}^{2}\rightarrow \mathbb{R}$ defined by%
\begin{equation*}
M\left( u,v\right) =\lambda _{\alpha }\left( u\right) u+\left( 1-\lambda
_{\alpha }\left( u\right) \right) v\text{, \ \ \ \ \ }M\left( u,v\right)
=\left( 1-\lambda _{\alpha }\left( u\right) \right) u+\lambda _{\alpha
}\left( u\right) v\text{,}
\end{equation*}%
are means in $\mathbb{R}$;

(ii) \ the means $M$,$N$ are continuous only at the points of the diagonal $%
\Delta :=\left\{ \left( x,x\right) :x\in \mathbb{R}\right\} $ and 
\begin{equation*}
\left\vert M\left( u,v\right) -N\left( u,v\right) \right\vert \leq \kappa
\left\vert u-v\right\vert \text{, \ \ \ \ \ }u,v\in \mathbb{R}\text{;}
\end{equation*}

(iii) the arithmetic mean $A:\mathbb{R}^{2}\mathbb{\rightarrow R}$, $A\left(
u,v\right) =\frac{u+v}{2}$, is a unique $\left( M,N\right) $-invariant and 
\begin{equation*}
\lim_{n\rightarrow \infty }\left( M,N\right) ^{n}=\left( A,A\right) \text{ \
\ \ \ (pointwise).}
\end{equation*}
\end{proposition}

\begin{remark}
This result remains true on replacing $\alpha \left( u\right) $\ by\ $\alpha
\left( f\left( u,v\right) \right) $ where $f:\mathbb{R}^{2}\mathbb{%
\rightarrow R}$ is an arbitrary nonconstant regular function.
\end{remark}

\begin{remark}
It is not difficult to observe that the above proposition can be modified to
a result in which $A$ is replaced by an arbitrary quasi-arithmetic mean. \ 
\end{remark}

\section{An applications in solving a functional equation}

\bigskip Applying Theorem 1 we prove the following

\begin{theorem}
Assume that $\mathbf{M}\colon I^{p}\rightarrow I^{p}$ is a mean-type mapping
and $K:I^{p}\rightarrow I$ is its unique $\mathbf{M}$-invariant mean. A
function $F:I^{p}\rightarrow \mathbb{R}$ which is continuous on the diagonal $%
\Delta \left( I^{p}\right) :=\left\{ \left( u_{1},...,u_{p}\right) \in
I^{p}:u_{1}=...=u_{p}\right\} $ is invariant with respect to the mean-type
mapping, i.e. $F$ satisfies the functional equation%
\begin{equation}
F\circ \mathbf{M}=F,  \label{6}
\end{equation}
if, and only if, there is a continuous function $\varphi :I\rightarrow 
\mathbb{R}$ such that 
\begin{equation*}
F=\varphi \circ K\text{. }
\end{equation*}
\end{theorem}

\begin{proof}
Assume first that $F:I^{p}\rightarrow \mathbb{R}$ that is continuous on the
diagonal $\Delta \left( I\right) $ and $F$ satisfies \eqref{6}. From \eqref{6} by
induction we get%
\begin{equation*}
F=F\circ \mathbf{M}^{n},\text{ \ \ \ \ }n\in \mathbb{N}_{0}.
\end{equation*}%
By Theorem 1 the sequence of mean-type mappings $\left( \mathbf{M}%
^{n}\right) _{n\in \mathbb{N}_{0}}$ converges pointwise to the mean-type
mapping $\mathbf{K}=\left( K,...,K\right) :I^{p}\rightarrow I^{p}$. Since $F$
is continuous on the diagonal $\Delta \left( I\right) $, we hence get for
all $u=\left( u_{1},...,u_{p}\right) \in I^{p}$, \ 
\begin{eqnarray*}
F\left( u_{1},...,u_{p}\right) &=&\lim_{n\rightarrow \infty }F\left( \mathbf{%
M}^{n}\left( u_{1},...,u_{p}\right) \right) =F\left( \lim_{n\rightarrow
\infty }\left( \mathbf{M}^{n}\left( u_{1},...,u_{p}\right) \right) \right)\\
&=&F\left( \mathbf{K}\left( u_{1},...,u_{p}\right) \right) =F\left( \left( K\left( u_{1},...,u_{p}\right) ,...,K\left(
u_{1},...,u_{p}\right) \right) \right) ,
\end{eqnarray*}
whence, setting%
\begin{equation*}
\varphi \left( t\right) :=F\left( t,...,t\right) \text{, \ \ \ \ \ }t\in I\,%
\text{,}
\end{equation*}%
we obtain $F\left( u_{1},...,u_{p}\right) =\varphi \left( K\left(
u_{1},...,u_{p}\right) \right) $ for all $\left( u_{1},...,u_{p}\right) \in
I^{p}$, that is$\ F=\varphi \circ K$.

To prove the converse implication, take an arbitrary function $\varphi
:I\rightarrow \mathbb{R}$ and put $F:=\varphi \circ K$. Then we have 
\begin{equation*}
F\circ \mathbf{M=}\left( \varphi \circ K\right) \circ \mathbf{M}=\varphi
\circ \left( K\circ \mathbf{M}\right) =\varphi \circ K=F,
\end{equation*}%
which completes the proof.
\end{proof}

\begin{remark}
The assumption of the continuity of the restriction of the function$\ F$ on
the diagonal $\Delta \left( I^{p}\right) $ is essential.

To show it take arbitrary (not necessarily continuous) function $\varphi
:I\rightarrow \mathbb{R}$ and define $F:I^{p}\rightarrow \mathbb{R}$ by 
\begin{equation*}
F\left( u_{1},...,u_{p}\right) :=\varphi \left( t\right) \text{ \ \ if \ \ }%
\lim_{n\rightarrow \infty }\mathbf{M}^{n}\left( u_{1},...,u_{p}\right)
=\left( t,...,t\right) .
\end{equation*}%
Since $\lim_{n\rightarrow \infty }\mathbf{M}^{n}\left(
u_{1},...,u_{p}\right) =\lim_{n\rightarrow \infty }\mathbf{M}^{n}\left( 
\mathbf{M}\left( u_{1},...,u_{p}\right) \right) ,$we have for all $u\in
I^{p} $, 
\begin{equation*}
F\left( u\right) =F\left( \mathbf{M}\left( u\right) \right) \text{.}
\end{equation*}
\end{remark}

\begin{remark}
If $F$ is a pre-mean then $\varphi=\id$ and consequently $F=K$.
Therefore if threre exists a uniquely determined $\mathbf{M}$-invariant mean then it is also the unique $\mathbf{M}$-invariant premean which is continuous on the diagonal.
% in view of Theorem 1, the continuity of $F$ is
 %superfluous (see \cite{JM2013}).
\end{remark}

\bigskip

\bigskip

\end{document}